\documentclass[10pt]{amsart}

\usepackage{amsmath, amsthm}
\usepackage{graphicx,epsfig}
\usepackage{amsmath,latexsym,amssymb,verbatim}

\newtheorem{definition}{Definition}
\newtheorem{lemma}{Lemma}
\newtheorem{proposition}{Proposition}
\newtheorem{theorem}{Theorem}
\newtheorem{corollary}{Corollary}

\newtheorem{example}{Example}

\newtheorem*{flowbox}{The \emph{flow-box} theorem}

\def\sideremark#1{\ifvmode\leavevmode\fi\vadjust{\vbox to0pt{\vss 
      \hbox to 0pt{\hskip\hsize\hskip1em           
 \vbox{\hsize2cm\tiny\raggedright\pretolerance10000
 \noindent #1\hfill}\hss}\vbox to8pt{\vfil}\vss}}}%

\begin{document}

\title[box dimension, saddle loop]{Box dimension of a hyperbolic saddle loop}
\address{University of Zagreb, Faculty of Electrical Engineering and Computing, Department of Applied Mathematics, Unska 3, 10000 Zagreb, Croatia}
\email{maja.resman@fer.hr}

\author{Maja Resman}
\thanks{This work was partially supported by the Croatian Science Foundation project IP-2014-09-2285.}

\begin{abstract}
We compute the box dimension of a spiral trajectory around a hyperbolic saddle loop, as the simplest example of a hyperbolic saddle polycycle. In cases of weak foci and limit cycles, \v Zubrini\' c and \v Zupanovi\' c show that the box dimension of a spiral trajectory is in a bijective correspondence with cyclicity of these sets. We show that, in saddle loop cases, the box dimension is related to the cyclicity, but the correspondence is not bijective.

In addition, complex saddles are complexifications of weak foci points, as well as of hyperbolic saddles. Computing the box dimension around the saddle point of a hyperbolic saddle loop is hopefully a preliminary technique for computing the box dimension of leaves of a foliation around resonant complex saddles.
\end{abstract}

\maketitle
Keywords: box dimension, hyperbolic saddle loop, cyclicity

MSC 2010: 37C10, 28A75, 37C27, 37C29

\bigskip 

\section{Introduction}\label{one}

\subsection{Motivation}\label{oneone}
Elementary \emph{limit periodic sets} include weak foci, limit cycles and hyperbolic saddle polycycles. For a good overview of what is known about the cyclicities of these sets, see the book of Roussarie \cite{roussarie}.

The overall use of fractal dimensions in dynamics is explained in short in e.g. \cite{encikl}. One of the fractal dimensions is the \emph{box dimension}. There is a relation between the box dimension and the cyclicity of limit periodic sets, noted by a group of authors in \cite{neveda}, \cite{buletin}, \cite{belgproc}. It was shown that the \emph{box dimension} of either a spiral trajectory in the plane or of a discrete orbit of the Poincar\' e map on the real line uniquely reveals the cyclicity of a limit cycle or a weak focus (in generic unfoldings). Box dimension of a trajectory is closely related to the asymptotic behavior of the Lebesgue measure of the $\varepsilon$-neighborhood of the trajectory, as $\varepsilon\to 0$, see Section~\ref{onetwo} below for the precise definition. Only one (any) spiral trajectory is needed. 

In \cite{belgproc}, box dimension of spiral trajectories around \emph{weak foci} and \emph{limit cycles} was computed. The box dimension was computed applying either \emph{the flow-box theorem} \cite{dumortier} or \emph{the flow-sector theorem} \cite{belgproc}, which state a locally parallel structure of the trajectory at a transversal to the set. Box dimension is computed as Cartesian product dimension, using the results about the box dimension of the one-dimensional orbit of the Poincar\' e map given in \cite{neveda}. The box dimension of a spiral trajectory is thus brought in a bijective correspondence with cyclicity in generic bifurcations.

Our study here is concerned with relating geometric, fractal properties of one trajectory around the loop with the cyclicity of the loop. In this paper, we compute the box dimension of a spiral trajectory around simplest hyperbolic saddle polycycles, hyperbolic loops, and compare it with cyclicity results. The descripton of bifurcations and the cyclicity of homoclinic sets, including hyperbolic saddle loops, is of interest in many applications, see the book of Palis, Takens \cite{PT}. In \cite{roussarie} or \cite{joyal}, the Poincar\' e maps of hyperbolic saddle polycycles are described as having an expansion in a Chebyshev power-log scale. The cyclicity was related to the first non-zero coefficient in the expansion called the \emph{loop quantity}, which is the analogon of Lyapunov coefficient for focus points. 

In cases of limit cycles and foci, the Poincar\' e map has the same asymptotics on every transversal to the set. Here we show that the Poincar\' e map at any transversal through the saddle vertex has a different asymptotics than Poincar\' e maps at other transversals to the loop. Therefore, to compute box dimension of the spiral trajectory, we develop a new version of the flow-box theorem that \emph{identifies} the trajectory around a saddle point with an appropriate \emph{set of hyperbolas}. 

Furthermore, the Poincar\' e map of limit cycles and foci is differentiable at zero. That is, their asymptotic expansion is a formal Taylor series. The Poincar\' e maps of saddle polycycles expand in nondifferentiable real power-log Chebyshev scales, see \cite{roussarie}, or \cite{mardesic} for Chebyshev scales. We show here that, unlike in differentiable cases from \cite{belgproc}, in nondifferentiable saddle-loop cases the box dimension is not sufficient for recognizing cyclicity. By its definition, the box dimension compares the area of the $\varepsilon$-neighborhood to a power scale. It is implicitely shown in this paper that, in nondifferentiable cases, the areas of the $\varepsilon$-neighborhoods of trajectories are not of power-type behavior. To get precise information, we need beforehand a finer logarithmic scale, depending on the bifurcation.
The box dimension of a whole trajectory thus reflects the same deficiency as the box dimension of one-dimensional orbits of the Poincar\' e map treated in \cite{cheby}: it cannot recognize between two neighboring cyclicities.

\smallskip
We expect future applications of the result to be twofold. First, a similar technique can be adapted to computing the box dimension of spiral trajectories around \emph{degenerate and nilpotent foci}. As in the saddle loop case, the asymptotics of the Poincar\' e map is dependent on the choice of the transversal. That is, there exist at least two distinct asymptotics.

Secondly, planar fractal analysis obviously distinguishes between\linebreak strong and weak foci, see \cite{buletin}, \cite{cheby}, but fails in distinguishing between hyperbolic saddle points. Namely, such points are not monodromic, trajectories are not recurring and the Poincar\' e map is not well-defined. Resonant complex saddles are complexifications of weak foci points, as well as of hyperbolic saddles. They are monodromic. Their monodromy maps are called \emph{holonomy maps}, see e.g. \cite{ilya},\ \cite{loray}. The holonomy maps are complex parabolic germs defined on two-dimensional transversals. The box dimension of such germs was computed in \cite{resman}. By \cite{ilya}, the structure is again locally parallel, except at the saddle vertex. Computing the box dimension of planar hyperbolic saddle loops is hopefully a preliminary technique for computing the box dimension of a leaf of a foliation around a complex saddle point. The box dimension is expected to reveal the formal type of the complex saddle. For our detailed conjecture on the box dimension of a complex resonant saddle and its relation to saddle formal invariants, see \cite[Chapter 2]{thesis}.
\smallskip

\emph{Overview of the article. }In Section~\ref{one}, we state definitions and known results that we will use in the article. Section~\ref{main} is dedicated to the two main results, Theorem~\ref{saddlepoinc} and Theorem~\ref{codim}. Theorem~\ref{saddlepoinc} describes the asymptotic behaviour of the Poincar\' e map on a transversal through the hyperbolic loop vertex. Theorem~\ref{codim} gives the box dimension of a hyperbolic saddle loop. The results are proven in Section~\ref{three}. In Section~\ref{four}, we discuss two main applications. Subsection~\ref{fourone} explains the relation between the box dimension and cyclicity for a hyperbolic saddle loop. In Subsection~\ref{fourtwo}, we give a conjecture on the box dimension of a leaf of a foliation at a resonant complex saddle, related to the first formal invariant of the saddle.

\subsection{Definitions and notation}\label{onetwo}
Let $X$ be an analytic planar vector field with a monodromic hyperbolic saddle connection. That is, two eigenvalues of the linear part at the saddle point are of the opposite sign. After some rotation/translation changes of variables and possibly rescaling the time variable we can, without loss of generality, suppose that the loop is attracting, that the saddle lies at the origin and its separatrices correspond to the coordinate axes. The operations mentioned above do not change the phase portrait (only rotate and translate the loop), and therefore the box dimension of trajectories remains the same. Locally at the origin, the vector field is of the following form:
\begin{align*}
\begin{cases}
\dot x&= x+P(x,y),\\
\dot y&=-r\cdot y+Q(x,y).
\end{cases}
\end{align*}
Here, \footnote{In the case $r<1$, the loop is \emph{repelling}, so we rescale the time variable $t\leftrightarrow -rt$ and change the role of $x$ and $y$. In this way, we get an \emph{attracting} loop with $r>1$ and the phase portrait remains the same.}$r\geq 1$ is the hyperbolicity ratio and $P$,\ $Q$ are analytic functions of higher order than linear. 

If $r\in \mathbb{Q}_+^*$, the saddle is called \emph{resonant}. We put $r=p/q$, where $p,\ q\in\mathbb{N}$, $(p,q)=1$. In the case $r\in\mathbb{R_+^*}\setminus \mathbb{Q}$, the saddle is \emph{nonresonant}.

\smallskip
Let $\tau_1\equiv \{y=1\}$ and $\tau_2\equiv \{x=1\}$ represent the horizontal and the vertical transversal, parametrized so that the origin lies on the loop. Let $P(s)$ denote the Poincar\' e map on any transversal $\tau$ not passing through the origin. The asymptotics (as $s\to 0$) of Poincar\' e map is well-known and can be found in e.g. \cite[Sections 5.1.3,\ 5.2.2]{roussarie}\footnote{The notation $\sim$ stands for: $$
f\sim g,\text{ as $x\to 0$, if}\lim_{x\to 0}\frac{f(x)}{g(x)}=L,\ L\neq 0.$$}:
\begin{equation}\label{poinc}
\begin{cases}
P(s)-s\sim s^k
\text { or }s^k(-\log s),\ k\geq 2,\ k\in\mathbb{N},&\text{ if $r=1$},\\
P(s)\sim s^r,&\text{ if $r>1$}.\end{cases}
\end{equation}

The Poincar\' e map $s\mapsto P(s)$ is a composition of the Dulac map $s\mapsto D(s)$ around the saddle and the regular transition map around the loop $s\mapsto R(s)$. The leading term, if logarithmic $s^k(-\log s)$, stems from the Dulac map, whereas the leading regular term $s^k$ stems from the regular loop transition map.

The notion of \emph{codimension} of the saddle loop is taken from \cite[Definition 27]{roussarie}, as the number of the conditions imposed on the loop.
In the case $r=1$, the saddle loop is said to be of \emph{codimension $2k$} if $s-P(s)\sim s^k$, and of \emph{codimension $2k+1$} if $s-P(s)\sim s^{k+1}(-\log s)$, for some $k\geq 1$. If $r\neq 1$, the saddle loop is of \emph{codimension $1$}.
\smallskip

In our computations in Section~\ref{main}, we use \emph{the flow-box theorem} from \cite[p.75]{kuzne} or \cite{dumortier}:

\begin{flowbox}
Let us consider a planar vector field of class $C^1$. Assume that $U\subset\mathbb{R}^2$ is a closed set the boundary of which is the union of two trajectories and two curves transversal to trajectories. If $U$ is free of singularities and periodic orbits, then the vector field restricted to U is diffeomorphically conjugated to the field:
\begin{align*}
\begin{cases}
\dot{x}=1,\\
\dot{y}=0,
\end{cases}
\end{align*}
on the unit square $\{(x, y) |\ 0\leq x,y\leq 1 \}$. That is, the flow on $U$ can be represented as a parallel flow.
\end{flowbox}

We close the section with definitions of lower and upper box dimension and some properties used in the article. Let $U\subset\mathbb{R}^N$ and its $\varepsilon$-neighborhood, denoted $U_\varepsilon$, both be Lebesgue measurable. We denote the Lebesgue measure of the $\varepsilon$-neighborhood by $|U_\varepsilon|$. The \emph{lower box dimension} $\underline\dim_B(U)$ of $U$ is defined as
\begin{align*}
\underline \dim_B(U)=\inf\{s\geq 0\,|\,\liminf_{\varepsilon\to 0}\frac{|U_\varepsilon|}{\varepsilon^{N-s}}=0\}.
\end{align*}
Analogously, we define the \emph{upper box dimension} of $U$, $\overline \dim_B(U)$, with $\limsup$ in the above formula.
If $\underline\dim_B(U)=\overline\dim_B(U)$, then $$\dim_B(U)=\underline\dim_B(U)=\overline\dim_B(U)$$ is called the \emph{box dimension} of $U$.

The upper and the lower box dimension satisfy the following properties that can be looked up in e.g. \cite{tricot} or \cite{falconer}. The upper (lower) box dimension of the \emph{Cartesian product} is the sum of the box dimensions of factors. The upper box dimension satisfies \emph{the finite stability property}: the box dimension of the finite union is equal to the biggest of dimensions. Furthermore, they both satisfy \emph{monotonicity property}: the dimension of a subset is smaller than or equal to the dimension of the whole set. The upper (lower) box dimension of an image of a set under a \emph{Lipschitz map} is smaller than or equal to the box dimension of the original. In particular, box dimension is invariant to \emph{bilipschitz transformations}.

\section{Main results}\label{main}
In Theorem~\ref{codim} we state the main result, the box dimension of a hyperbolic saddle loop. Theorem~\ref{saddlepoinc} and Lemmas~\ref{fh} and \ref{dim} are used in the proof, but they are also independent results. All proofs are in Section~\ref{three}.
\smallskip

Recall the asymptotics of the Poincar\' e map on a \emph{transversal to the hyperbolic saddle loop not passing through the saddle} from \cite{roussarie}, see \eqref{poinc} in Section~\ref{one}. The asymptotics by definition determines the box dimension of the orbit. Compare it to Theorem~\ref{saddlepoinc}. In \cite{cheby}, it was shown that the box dimension of an orbit of the Poincar\' e map on a transversal away from the vertex does not reveal the codimension of the loop and the cyclicity uniquely. Here we compute the asymptotics on a transversal \emph{through the saddle}. We see that, although the scale is slightly different, the information carried in the asymptotics is the same.  

\begin{theorem}[Poincar\' e map on a transversal through the saddle]\label{saddlepoinc}
Let $\tau$ be a transversal through the saddle of a hyperbolic saddle loop of codimension $k\in\mathbb{N}$. Then the Poincar\' e map $s\mapsto P(s)$ defined on $\tau$ has the following asymptotics, as $s\to 0$:
\begin{equation}\label{lp}
s-P(s)\sim\begin{cases}
s^r,\ r>1 \text{ the ratio of hyperbolicity},& \text{if $k=1$,}\\
s^{k-1},&\text{if $k$ even,}\\
s^{k}(-\log s),&\text{if $k$ odd, $k>1$.}
\end{cases}
\end{equation}
\end{theorem}

\begin{corollary}\label{kor}
The box dimensions of orbits of the Poincar\' e maps on a \emph{transversal away from the saddle vertex} and \emph{through the vertex} are equal to, respectively,
$$
\begin{array}{ll}
\dim_B =\begin{cases}
1-\frac{2}{k}, & \text{$k$ even,}\\
1-\frac{2}{k+1}, & \text{$k$ odd}.
\end{cases}
&\dim_B =\begin{cases}
1-\frac{1}{k-1}, & \text{$k$ even,}\\
1-\frac{1}{k}, & \text{$k$ odd}.
\end{cases}
\end{array}
$$
\end{corollary}

Along the loop, we have two distinct directions with different rates of growth of Poincar\' e maps: transversals through and not through the saddle. The asymptotics is slower through the saddle. This is new, compared to the focus and limit cycle case treated in \cite{belgproc}, and similar to degenerate and weak foci cases \cite{arxivlana}. We check here if the saddle influences the box dimension of the trajectory, that is, if the accumulation of density of the trajectory lies at the saddle or away from the saddle. 
\smallskip

Let $x_0$ be an initial point lying \emph{close} to the loop and let $S(x_0)$ denote the spiral trajectory with the initial point $x_0$, accumulating at the loop.
\begin{theorem}[Box dimension of the spiral trajectory around a saddle loop]\label{codim}
Let $k\geq 1$ be the codimension of the saddle loop. Then
$$
\dim_B\big(S(x_0)\big)=\begin{cases}
2-\frac{2}{k},& \text{$k$ even},\\
2-\frac{2}{k+1},& \text{$k$ odd}.\end{cases}
$$
\end{theorem}
Note that the dimension of a trajectory from Theorem~\ref{codim} is the product dimension \emph{around any transversal not passing through the saddle} (by Flow-box theorem, see Section~\ref{one}). Given the finite stability property of box dimension, the accumulation at the saddle is obviously not dense enough to carry the box dimension.
\smallskip

\emph{Sketch of the proof.} We compute the box dimension dividing the trajectory in two parts, $S(x_0)=S_1(x_0)\cup S_2(x_0)$, and using the finite stability property of box dimension. $S_1(x_0)$ is the non-regular part of the spiral trajectory near the saddle, between the transversals $\{x=1\}$ and $\{y=1\}$ (a family of hyperbolas). $S_2(x_0)$ is the remaining regular part of the trajectory. We first compute the box dimension of $S_1(x_0)$. The main tools are Lemma~\ref{fh} and Lemma~\ref{dim} stated below. First, in Lemma~\ref{fh}, we show that $S_1(x_0)$ can, by a bilipschitz mapping, be transformed to a simpler\footnote{see Definition \ref{para}} family of hyperbolas, intersecting the transversals at the points with the same asymptotics as the original family of hyperbolas. Then, in Lemma~\ref{dim}, we compute the box dimension of a family of hyperbolas parametrized by a discrete set with the known box dimension on a transversal. To compute the box dimension of the regular part $S_2(x_0)$, we can directly apply the \emph{Flow-box theorem} from Section~\ref{one}. In detail, the proof is given in Section~\ref{three}.
\medskip

The statement of the following Lemma~\ref{fh} is the \emph{analogon} of the flow-box theorem, but for a flow near a singular point whose portrait is a family of hyperbolas. It concerns \emph{straightening} the original family to a simpler family of hyperbolas.
\begin{lemma}[\emph{Orbital} linearization]\label{fh}
Let $X$ be an analytic vector field with a saddle at the origin:
\begin{align}\label{field}
\begin{cases}
\dot{x}&=x+P(x,y),\\
\dot{y}&=-r\cdot y +Q(x,y),\ r>0.
\end{cases}
\end{align}
Then there exists a neighborhood of the origin and a local diffeomorphism acting quadrant-wise, that transforms the phase portrait of \eqref{field} to the phase portrait of its linear part. 
\end{lemma}
Recall the $C^r$-normal forms for hyperbolic saddles stated in \cite[Theorem 13]{roussarie}. The field \eqref{field} above is not necessarily $C^1$-linearizable in the sense of $C^1$-conjugacy of fields, but its phase portrait is, by Lemma~\ref{fh}, $C^1$-equivalent to the phase portrait of its linear part. It is a weaker statement than $C^1$-conjugacy of fields, see e.g. \cite{ilya} or \cite{teyssier} for precise definitions.
\medskip 

\begin{definition}[Family $\mathcal H_{r,S}$]\label{para}
Let $r>0$. By $$\mathcal{H}_{r,S}=\{(x,y)\in (0,1]\times (0,1]\ |\ x^r y=c,\ c\in S\},\ S\subset\mathbb R,$$ we denote a family of hyperbolas \footnote{a set of \emph{level curves} of function $F(x,y)=x^r y$} parametrized by $S\subset\mathbb{R}$. 
\end{definition}
Note: The family $\mathcal H_{r,S}$ in Definition~\ref{para} consists of \emph{integral curves of a linear field} in the first quadrant:
\begin{align*}
\begin{cases}
\dot{x}&=x,\\
\dot{y}&=-r\cdot y.
\end{cases}
\end{align*}
Note that the set $S$ and $S^{1/r}=\{c^{1/r}\,|\,c\in S\}$ respectively consist of points on the transversal $\{x=1\}$ and $\{y=1\}$ respectively, where the family of hyperbolas $\mathcal H_{r,S}$ intersects those transversals.

\begin{lemma}[Box dimension of $\mathcal H_{r,S}$]\label{dim}
Let \footnote{Every mention of the interval $(0,1]$ can be replaced by interval $(0,\delta]$, $\delta>0$. The analysis is local, in a neighborhood of the singular point (the origin). Here, $(0,1]\times(0,1]$ is taken just for simplicity.}$S\subset\mathbb (0,1]$ be a sequence of points accumulating at the origin, with \footnote{except at most finitely many first distances} distances between the points eventually decreasing.
The box dimension of the family of hyperbolas $$\mathcal{H}_{r,S}=\{(x,y)\in(0,1]\times (0,1]\ |\ x^r y=c,\ c\in S\}$$ is equal to
\begin{align*}\label{flowhyp}
\dim_B(\mathcal{H}_{r,S})&=\begin{cases}1+\dim_B S,& r\leq 1,\\1+\dim_B (S^{1/r}),& r>1.\end{cases}
\end{align*}
\end{lemma}

Note that, by Lemma~\ref{dim}, the box dimension of the family $H_{r,S}$ is in fact the bigger of the two box dimensions of parallel flows at transversals $\{x=1\}$ and $\{y=1\}$. The accumulation at the saddle point is not what prevails in box dimension. The accumulation of density is not at the vertex.
\medskip

We illustrate the product statement of Lemma~\ref{dim} on Figure~\ref{hipi} below. In the figure, $r>1$. The family of hyperbolas $\{x^r y=c,\ c\in S\}$, intersects the transversals $\{x=1\}$ and $\{y=1\}$ in one-dimensional \emph{sequences} $S$ and $S^{1/r}$ respectively. Box dimension of the sequence $S^{1/r}$ is bigger than the box dimension of the sequence $S$ ($x^r$ for $r>1$ is a Lipschitz map). The accumulation of density is therefore around the horizontal transversal $\{y=1\}$, and the set of hyperbolas takes the product box dimension around $\{y=1\}$, that is,
$$
\dim_B(\mathcal{H}_{r,S})=\dim_B (S^{1/r}) +1.
$$
\begin{figure}[h]
\begin{center}
  \includegraphics[width=3in]{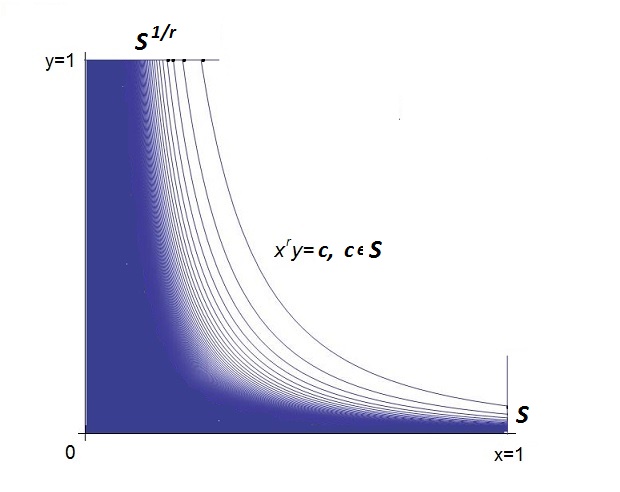}
  \caption{\small{Family $\mathcal{H}_{r,S}$ of hyperbolas in Lemma~\ref{dim}, $r>1$.}}\label{hipi}
  \vspace{0.5cm}
  \end{center}
\end{figure}
\smallskip

In Example~\ref{staro}, we give examples of sequences $S$ generated by Poincar\' e maps on transversals of hyperbolic saddle loops.
\begin{example}\label{staro}
Let $g:(0,\delta]\to (0,\delta],$ $\delta>0$, and let $f=id-g$. For $0<s_0<\delta$, we consider the \emph{orbit} of $g$ with initial point $s_0$:
$$S=\{s_n=g^{\circ n}(s_0)|\ n\in\mathbb{N}_0\}\subset(0,\delta].$$ In cases $(1)$ and $(2)$ below, the fixed point $s=0$ is \emph{parabolic}, and in case $(3)$ it is \emph{weakly or strongly hyperbolic.} A saddle loop with the ratio of hyperbolicity of the saddle $r>1$ has Poincar\' e maps of strongly hyperbolic type. If the ratio is $r=1$, the Poincar\' e maps are either of parabolic or of weakly hyperbolic type, see \cite[Chapter 5]{roussarie}. 
\begin{enumerate}
\item Let
 $$
 f(s)\simeq s^{\alpha},\ s\to 0;\ \ \alpha\geq 1.
 $$
 By \cite[Theorem 1]{neveda},
 \begin{equation*} s_n\simeq n^{-\frac{1}{\alpha-1}},\ n\to\infty,\quad \dim_B S=1-\frac{1}{\alpha}.\end{equation*}
 \item Let
 $$f(s)\simeq s^{\alpha}(-\log s),\ s\to 0;\ \ \alpha>1.$$
By \cite[Theorem 2]{cheby}, $$\dim_B S=1-\frac{1}{\alpha}.$$
\item Let \begin{align*} g(s)&=k s+o(s),\ 0<k<1, \text{\ \ or\ \ \quad} \\
g(s)&=C s^\beta+o(s^\beta),\ \beta>1,\ C>0,\ \ s\to 0.\end{align*}
The orbit $S$ accumulates at zero exponentially fast. There exist $\gamma\in(0,1)$ and $C>0$ such that
\begin{equation*}
0<s_n<C \gamma^n. 
 \end{equation*}
By \cite[Lemma 1, Theorem 5]{neveda}, $$\dim_B S=0.$$
 \end{enumerate}
 \end{example}

\section{Proofs of main results}\label{three}

\noindent\emph{Proof of Theorem~\ref{saddlepoinc}.}
We mimic the proof of Poincar\' e asymptotics on a transversal not passing through the saddle from \cite[Section 5]{roussarie}, adapted for transversals through the saddle. Since the asymptotics is invariant to conjugacy, it suffices to work with finite $C^k$-normal forms of the vector field from \cite[Theorem 13]{roussarie}, $k\in\mathbb{N}$.

Let $\tau$ denote a transversal through the saddle. For simplicity, we take the diagonal $\tau\equiv\{y=x\}$. By $D_1$ and $D_2$, we denote the transition maps from the horizontal transversal $\{y=1\}$ to $\tau$ and from $\tau$ to the vertical transversal $\{x=1\}$. Then, $D=D_2\circ D_1$ is the \emph{saddle transition map}. Let $R$ denote the \emph{regular loop transition}. The Poincar\' e map on $\tau$ is given by the following composition:
$$
P=D_1\circ R\circ D_2.
$$

Let $r=1$. By \cite{roussarie}, after change of variables to $(x,u)$, $u=xy,$
\begin{equation}\label{rou}
\begin{cases}
\dot{x}&=x,\\
\dot{u}&=\sum_{i=2}^{\infty}\alpha_i u^i.
\end{cases}
\end{equation}
From \eqref{rou} it can be computed that $D(s)=u(-\ln s,s)$, see \cite[Section 5.1]{roussarie} for details. Similarly, we compute:
\begin{align}\label{fact}
&D_2(s)=u(-\ln\frac{s}{\sqrt 2},\frac{s^2}{2}),\\
&\frac{D_1(s)^2}{2}=u\left(\ln\frac{D_1(s)}{s\sqrt 2},s\right).\nonumber
\end{align}
By \cite[Theorem 14]{roussarie}, $$u(\omega,s)=s+a s^k \omega+h.o.t.,\ a\in\mathbb{R},\ k\in\mathbb N,\ k\geq 2.$$ Here, $\omega$ is a logarithmic monomial and the \emph{higher order terms} are meant with respect to a lexicographic order imposed on monomials $s^m \omega^n,\ m,\,n\in\mathbb{N}$. Hence, $$D(s)=s+as^k(-\log s)+h.o.t.$$

Suppose that $R(s)=a_1s+a_2 s^l+o(s^l)$ is of order $l$, $l\geq 2$. The Poincar\' e map $P_1=R\circ D$ on the transversal $\{y=1\}$ obviously has the asymptotics: \begin{equation}\label{c1}P_1(s)=a_1s+\footnote{Here, \emph{$min$} is meant in the sense: of the  lower order with respect to lexicographic order.}\min\{a_2s^l,aa_1s^k(-\ln s)\}+h.o.t.\end{equation}
Furthermore, by \eqref{fact}, we get:
\begin{align*}
&D_2(s)=\frac{s^2}{2}+\frac{a}{2^{k}} s^{2k}(-\ln s)+h.o.t.,\\[0.1cm]
&D_1(s)=\sqrt 2 s^{1/2}+\frac{a}{2\sqrt 2}s^{k-1/2}(-\ln s)+h.o.t.
\end{align*}
Componing $P=D_1\circ R\circ D_2$, we get that \begin{equation}\label{c2}P(s)=a_1s+\min\{a_2a_1 2^{-l}\cdot s^{2l-1},aa_1 2^{-k}\cdot s^{2k-1}(-\ln s)\}+h.o.t.\end{equation} Compare \eqref{c1} and \eqref{c2}, and the statement of the lemma for $r=1$ follows.
\medskip

Now let $r>1$. If the saddle is \emph{nonresonant} ($r\notin \mathbb Q$), it is by \cite[Theorem 13]{roussarie} locally $C^k$-linearizable, $k\in\mathbb{N}$. The linear system
$$
\begin{cases}
&\dot{x}=x,\\
&\dot{y}=-ry
\end{cases}
$$
has an explicit solution $x(t)=Ce^t$, $y(t)=De^{-rt}$. Using explicit formulas, we easily get that $$D(s)\sim s^r,\ D_1(s)\sim s^{\frac{r}{r+1}},\ D_2(s)\sim s^{r+1},\ s\to 0.$$ Therefore, $P_1(s)\sim s^r,\ P(s)\sim s^r,\ s\to 0$.

On the other hand, let the saddle be \emph{resonant} ($r>1$, $r\in\mathbb Q$). We put $r=p/q$. Similarly as in the first case, after the change of coordinates to $(x,u)$, $u=x^p y^q$, we compute $$u(\omega,s)=s+a s^k \omega+h.o.t.,\ k\in\mathbb N,\ k\geq 2.$$ Since $D(s)^q=u(-\ln s,s^p)=s^p+o(s^p)$, it follows that $D(s)\sim s^r,\ s\to 0$. We compute $$D_2(s)=u\Big(-\ln\frac{s}{\sqrt 2},\left(\frac{s} {\sqrt 2}\right)^{p+q}\Big)^{1/q}\sim s^{r+1},\ s\to 0.$$ From $D=D_2\circ D_1$, we have that  $D_1(s)\sim s^{\frac{r}{r+1}},\ s\to 0$. Thus, $$P_1(s)\sim s^r,\ P(s)\sim s^r,\ s\to 0.$$
\hfill $\Box$
\medskip

\noindent\emph{Proof of Corollary~\ref{kor}.} It follows directly from Theorem~\ref{saddlepoinc} and Example~\ref{staro}.\hfill$\Box$
\medskip

\noindent \emph{Proof of Lemma~\ref{fh}.} Let $X$ have a \emph{nonresonant} saddle at the origin, $r\notin \mathbb{Q}$. By \cite[Theorem 13]{roussarie}, such a vector field is linearizable in a neighborhood of the origin: it is even diffeomorphically conjugated to its linear part. This proves the statement in the nonresonant case. 

We prove here the more complicated, \emph{resonant} case. Although not necessarily linearizable,  we show that the phase portrait is quadrant-wise diffeomorphically equivalent to the phase portrait of the linear part. Let $X$ have a resonant saddle at the origin, with $r\in\mathbb{Q}_+^*$. Let $r=\frac{p}{q}$ with $p,\ q\in\mathbb{N},\ (p,q)=1$. By Theorem 13 in \cite{roussarie}, there exists an integer $N\in \mathbb{N}$, the coefficients $a_2,\ldots,a_{N+1}\in\mathbb{R}$, and a neighborhood of the origin, such that $X$ is diffeomorphically conjugated to the finite polynomial vector field
\begin{align}\label{polin}
\begin{cases}
\dot{x}&=x,\\
\dot{y}&=-r\cdot y+\frac{1}{q}\sum_{i=1}^{N}a_{i+1}\cdot (x^p y^q)^i\cdot y.
\end{cases}
\end{align}

We construct a diffeomorphism $F(x,y)$, acting quadrant-wise in a neighborhood of the origin, which sends phase curves of field \eqref{polin} to phase curves of its linear part. We show here the construction of $F^{I}(x,y)$ in the first quadrant. Afterwards we glue functions $F^{I,II,III,IV}$ constructed in each quadrant to a global diffeomorphism $F$ at the origin.

We proceed as in \cite[5.1.2]{roussarie} (a similar technique was used there for obtaining the Dulac map at the resonant saddle). We solve the system \eqref{polin} by substitution $u=x^p y^q$, whereas we get the system
\begin{align}\label{pomo}
\begin{cases}
\dot{x}&=x,\\
\dot{u}&=\sum_{i=2}^{N} a_i u^i.
\end{cases}
\end{align} 

Solving \eqref{pomo} by expanding $u(t,u_0)$ in series with respect to the initial condition $u_0$, we get that
\begin{equation}\label{ju}
u(t,u_0)=u_0+\sum_{i=2}^{N} g_i(t) u_0^{i}.
\end{equation}
The form of $g_i(t)$, $i\geq 2$, is described in Proposition 10 in \cite{roussarie}: $g_i(t)$ are polynomials in $t$, of degree at most $i-1$. Therefore, we easily obtain the bounds:
\begin{equation}\label{bd}
|g_i(t)|\leq C_i t^{i-1},\ |g_i'(t)|\leq D_i t^{i-2};\ \ C_i,\ D_i>0,\quad i=2,\ldots,N,
\end{equation}
for $t$ sufficiently big.

Let $\tau_1\equiv \{y=1\}$ be a horizontal transversal to the saddle. We should in fact consider the transversal at some small height $\delta>0$ instead at height $1$, but the computations are the same. 
Using \eqref{ju} and $u=x^p y^q$, we can now derive the formula for the phase curve of the cut-off field \eqref{polin}, passing through the initial point $(x(0),y(0))=(s,1)\in\tau_1$. We put $u_0=s^p$ in \eqref{ju} and solve $\dot x=x$. We get that $t=\log\frac{x}{s}$, and then, for the phase curve through $(s,1)$, we have the formula:
\begin{equation}\label{s1}
y=\frac{s^r}{x^r}\big(1+\sum_{i=2}^{N}g_i(\log\frac{x}{s})s^{p(i-1)}\big)^{1/q},\ s\leq x\leq 1.
\end{equation}
The phase curve of the linear part passing through $(s,1)$ is, on the other hand, given by
\begin{equation}\label{s2}
y=\frac{s^r}{x^r},\ s\leq x\leq 1.
\end{equation}

We now define the mapping $F^{I}$ of the first quadrant $(0,1]\times(0,1]$ to itself, sending phase curves of \eqref{polin} to phase curves of the linear part in the following manner. For any point $(x,y)\in(0,1]\times(0,1]$ close to the saddle,  there exists a unique phase curve of the linear field passing through it, and it is determined by the point $(s,1)$ on $\tau_1$. We consider the phase curve of \eqref{polin} passing through the same point $(s,1)$. $F^{I}(x,y)$ is then defined as the orthogonal projection of $(x,y)$ on this phase curve. More precisely, by \eqref{s1} and \eqref{s2}, we get the formula for $F^I(x,y)$:
\begin{align}\label{ef}
F^{I}(x,y)=&\left(\ x,\ y\cdot \Big(1+\sum_{i=2}^{N}g_i\big(\log(y^{-1/r})\big)\cdot x^{p(i-1)}\cdot y^{q(i-1)}\Big)^{1/q} \right),\\ &\hspace{6.5cm} (x,y)\in(0,1]\times(0,1].\nonumber
\end{align}
Function $F^{I}$ is obviously well-defined, continuous and differentiable in $(0,1]\times(0,1]$. 
We can obtain similar formulas for $F^{II,III,IV}$ in other quadrants. In the first and in the second quadrant, in \eqref{s2}, we have
$y=s^r/x^r$. In the third and the fourth quadrant, we have $y=-s^r/x^r$. By glueing the quadrants together, we get function $F$, defined on the unit square without the coordinate axes, by the formula
\begin{align}\label{eff}
F(x,y)=\bigg(x,\,y&\Big(1+\sum_{i=2}^{N}g_i\big(\log(|y|^{-1/r})\big)\,x^{p(i-1)}y^{q(i-1)}\Big)^{1/q}\bigg),\nonumber\\
&(x,y)\in\big([-1,1]\times[-1,1]\big)\setminus \big(\{x=0\}\cup \{y=0\}\big).
\end{align}
It can be checked that the functions can be extended to the coordinate axes to a continuously differentiable function $F(x,y)$ defined on the square $[-1,1]\times [-1,1]$. 
Then we apply the inverse function theorem at the origin. We conclude that $F$ is a local diffeomorphism. Moreover, by construction, it leaves the axes invariant and maps each quadrant to itself.

Let us extend $F$ from \eqref{eff} to the coordinate axes, and check that the obtained function is continuously differentiable in $[-1,1]\times[-1,1]$. We define $$F(x,0):=(x,0),\ x\in[-1,1],\ \ F(0,y):=(0,y),\ y\in[-1,1].$$ Note that $F$ cannot be extended continuously simply by formula \eqref{ef} to $y=0$ due to the logarithmic term. However, using \eqref{bd}, $\lim_{y\to 0}F(x,y)=0,\ x\in\mathbb R$. Therefore, extended as above, $F$ is continuous on $[-1,1]\times [-1,1]$. 

Furthermore, $F$ given by \eqref{eff} is differentiable on $\big([-1,1]\times[-1,1]\big)\setminus \big(\{x=0\}\cup \{y=0\}\big)$. We can show, by direct computation of the differential and using bounds \eqref{bd}, that:
\begin{align}
&DF(x,y)=\left[\begin{array}{cc}1& 0\\\partial_x F_2(x,y)&\partial_y F_2(x,y)\end{array}\right],\ x\neq 0,\ y\neq 0,\label{limi}\\[0.3cm]
&\ \lim_{y\to 0} \partial_y F_2(x,y)=1,\ \ \lim_{x\to 0}\partial_y F_2(x,y)=1,\nonumber\\
&\ \lim_{y\to 0} \partial_x F_2(x,y)=0,\ \ \lim_{x\to 0}\partial_x F_2(x,y)=G(y).\nonumber
\end{align}
Here, 
$$
G(y)=\begin{cases}
\frac{1}{q}\cdot y^{q+1} \cdot g_2(\log |y|^{-q}),& p=1,\\
0,& p>1.
\end{cases}
$$

Obviously, $G(y)\to 0$ as $y\to 0$. We can check directly by definition of differentiability at $(x,0)$ and $(0,y)$, using bounds \eqref{bd} and \eqref{limi}, that $F$ extended to the axes in the above manner is continuously differentiable at the axes and that the differential is given by 
\begin{equation}\label{differe}
DF(x,y)=\begin{cases}\text{\eqref{limi}},&(x,y)\in\big([-1,1]\times[-1,1]\big)\nonumber\\
&\qquad \qquad \quad \setminus \big(\{x=0\}\cup \{y=0\}\big),\\[1mm]
\left[\begin{array}{cc}1& 0\\0&1\end{array}\right],&x\in[-1,1],\ y=0,\\[5mm]
\left[\begin{array}{cc}1& 0\\G(y)&1\end{array}\right],&x=0,\ y\in[-1,1].
\end{cases}
\end{equation}
Note in the course of the proof that $C^1$ is the best class that we can obtain applying bounds \eqref{bd}.
\hfill $\Box$

\medskip

\noindent\emph{Proof of Lemma~\ref{dim}}. The idea is similar to the proof of the box dimension of a Cartesian product in the plane. There, the area of the $\varepsilon$-neighborhood of a $2$-dimensional product $U\times [0,1]$ can be expressed using the length of the $\varepsilon$-neighborhood of the $1$-dimensional set $U$. 

Let $r\leq 1$. The family $\mathcal{H}_{r,S}$ is a family of phase curves (in the first quadrant) of the linear saddle with ratio of hyperbolicity $r$. By \emph{Flow-box theorem} from Section~\ref{one}, we get that the box dimension of $\mathcal{H}_{r,S}$ on a small rectangle around $\{x=1\}$ is equal to $1+\dim_B S$. Around $\{y=1\}$, the box dimension is equal to $1+\dim_B (S^{1/r})$. Since $y=x^{1/r}$, $r\leq 1$, is a Lipschitz map, we have that $\dim_B (S^{1/r})\leq \dim_B S$. By the monotonicity property of box dimension, we get
\begin{equation}\label{prod}
\underline{\dim_B}(\mathcal{H}_{r,S}),\ \overline{\dim_B}(\mathcal{H}_{r,S})\geq 1+\dim_B S.
\end{equation}

If $r=1$, the hyperbolas are symmetric with respect to the diagonal $y=x$. If $r<1$, we first \emph{symmetrize} the family by  the change $u=x^r,\ v=y$, which is inverse-Lipschitz.  In the new coordinate system $(u,v)$, we get the symmetric family of hyperbolas $\mathcal{H}_{1,S}=\{uv=c|\ c\in S\}$. We compute the box dimension of the \emph{symmetrized} family $\mathcal{H}_{1,S}$, and conclude that $$\dim_B\mathcal{H}_{r,S}\leq \dim_B\mathcal{H}_{1,S}.$$

The box dimension of $\mathcal{H}_{1,S}$ is, by symmetry and by finite stability property of box dimension, equal to the box dimension of its subset between the diagonal $y=x$ and the transversal $\{x=1\}$.

We now estimate the area of the $\varepsilon$-neighborhood of $\mathcal{H}_{1,S}$ in this area, establishing the almost product relation with the length of the $\varepsilon$-neighborhood of one-dimensional set $S$ on the transversal $\{x=1\}$. The intersection points of $\mathcal{H}_{1,S}$ and the diagonal $\{y=x\}$ form the set $\sqrt 2 S^{1/2}$ on the diagonal. The distances of the intersecting points on the transversal $\{x=1\}$ are smaller than the distances of corresponding points on any other transversal in the area, including the diagonal (at least close to $0$, considered on $(0,\delta]\times (0,\delta]$, $\delta>0$).

We compute the area dividing the $\varepsilon$-neighborhood into tail and nucleus. By \emph{tail} of the $\varepsilon$-neighborhood, $T_\varepsilon^2$, we mean the disjoint neighborhoods of the first finitely many hyperbolas. The remainder of the $\varepsilon$-neighborhood, where the neighborhoods of hyperbolas start overlapping, is called the \emph{nucleus}, and denoted $N_\varepsilon^2$. For the idea of division in tail and nucleus, see \cite{tricot}. Since the distances are the smallest on transversal $\{x=1\}$, the critical index separating the tail and the nucleus is the same as for one-dimensional set $S$ on $\{x=1\}$. Let $T_\varepsilon^1$, $N_\varepsilon^1$ denote the tail and the nucleus (in dimension one) of its $\varepsilon$-neighborhood. We now bring $T_\varepsilon^1$ and $T_\varepsilon^2$ and $N_\varepsilon^1$ and $N_\varepsilon^2$ into direct relation, and express the box dimension of $\mathcal{H}_{1,S}$ by the box dimension of $S$.

Since the lengths of hyperbolas are bounded from above and below in $(0,\delta]\times (0,\delta]$, we get that\footnote{The notation $\simeq$ stands for: $f\simeq g$,\ as $x\to 0$, if there exist $A,\ B>0$ and $\delta>0$, such that $$A\leq\frac{f(x)}{g(x)}\leq B,\ 0<x\leq \delta.$$} 
\begin{equation}\label{svv}
A(T_\varepsilon^2)
\simeq|T_\varepsilon^1|+\varepsilon\cdot |T_\varepsilon^1|\simeq |T_\varepsilon^1|,\ \varepsilon\to 0.
\end{equation}
By definition of box dimension, it holds that 
\begin{equation}\label{sv}
\lim_{\varepsilon\to 0}\frac{|T_\varepsilon^1|}{\varepsilon^{1-\dim_B S-\delta}}=0,\quad  \lim_{\varepsilon\to 0}\frac{|N_\varepsilon^1|}{\varepsilon^{1-\dim_B S-\delta}}=0,\ \text{ for all }\delta>0. 
\end{equation}
By \eqref{svv} and \eqref{sv}, we get 
\begin{equation}\label{prv}
\lim_{\varepsilon\to 0}\frac{A(T_\varepsilon^2)}{\varepsilon^{1-\dim_B S-\delta}}=\lim_{\varepsilon\to 0}\frac{A(T_\varepsilon^2)}{\varepsilon^{2-(\dim_B S+1)-\delta}}=0,\ \text{ for all }\delta>0.
\end{equation}

For the nucleus, we give an upper bound on the area of the $\varepsilon$-neighborhood. The hyperbola separating the tail and the nucleus has the equation $uv=|N_\varepsilon^1|$. The 
area $A(N_\varepsilon^2)$ is smaller than or equal to the area between the $u$-axis, the hyperbola $H_{n_\varepsilon}$, the diagonal $y=x$ and the transversal $\{x=1\}$:
\begin{align*}
A(N_\varepsilon^2)&\leq C|N_\varepsilon^1|+D\varepsilon+\int_{|N_\varepsilon^1|^{1/2}}^{1}\frac{|N_\varepsilon^1|}{y}\ dy,\ \varepsilon<\varepsilon_0.
\end{align*}
Take any fixed $\delta_0>0$. There exists a small $\nu>0$, such that $\delta_0-\nu>0$. Integrating, it holds that there exist $C>0$ and $\varepsilon_0$, such that
\begin{align*}
\frac{A(N_\varepsilon^2)}{\varepsilon^{1-\dim_B S-\delta_0}}&\leq \frac{C|N_\varepsilon^1|(-\log\varepsilon)}{\varepsilon^{1-\dim_B S-\delta_0}}\leq \frac{C|N_\varepsilon^1|}{\varepsilon^{1-\dim_B S-(\delta_0-\nu)}}\cdot\varepsilon^{\nu}(-\log\varepsilon),\ \ \varepsilon<\varepsilon_0.
\end{align*} 
Passing to limit as $\varepsilon\to 0$ in the above inequality and using \eqref{sv}, we get that
\begin{equation}\label{drug}
\lim_{\varepsilon\to 0}\frac{A(N_\varepsilon^2)}{\varepsilon^{1-\dim_B S-\delta}}=\lim_{\varepsilon\to 0}\frac{A(N_\varepsilon^2)}{\varepsilon^{2-(s+1)-\delta}}=0, \text{ for all } \delta>0.
\end{equation}
By \eqref{prv} and \eqref{drug}, it follows that $\underline{\dim_B}(\mathcal{H}_{1,S}),\ \overline{\dim_B}(\mathcal{H}_{1,S})\leq 1+s$ for the \emph{symmetrized} family. 

It follows that $\underline{\dim_B}(\mathcal{H}_{r,S})\leq \underline{\dim_B}(\mathcal{H}_{1,S})\leq 1+s.$ The same for the upper box dimension. Using \eqref{prod}, we finally get that, for $r\geq 1$,
$$
\dim_B(\mathcal{H}_{r,S})=1+s.
$$
 
Now let $r>1$. Then, $1/r<1$. We have
\begin{align*}
\mathcal{H}_{r,S}=&\{(x,y)\in(0,1]\times (0,1]\ |\ x^r y =c,\ c\in S\}=\\
=&\{(x,y)\in(0,1]\times (0,1]\ |\ y^{1/r} x =c,\ c\in S^{1/r}\}.
\end{align*}
Changing the roles of $x$ and $y$ and substituting $S$ for $S^{1/r}$, by the first part of the proof we get that
$$\dim_B (\mathcal{H}_{r,S})=1+\dim_B (S^{1/r}).$$ 
\hfill $\Box$

\medskip

\noindent\emph{Proof of Theorem~\ref{codim}}.
Let the saddle loop be of codimension $k$, $k\in\mathbb{N}$. Without loss of generality, suppose that the ratio of hyperbolicity is $r\geq 1$. The repelling case $r<1$ can be considered as attracting by rescaling the time variable $t\leftrightarrow -rt$, whereas the phase portrait (and so the box dimension of a trajectory) remains the same. Take a spiral trajectory $S(x_0)$ accumulating at the loop, with $x_0$ close to the loop. We divide the trajectory in two parts: the part $S_1(x_0)$ passing the saddle vertex, lying in the square $(0,1]\times(0,1]$, and the remaining regular part $S_2(x_0)$ around the saddle connection. 

We first compute $\dim_B(S_1(x_0))$. The asymptotics of the Poincar\' e map $s\mapsto P(s)$ on the transversal $\{x=1\}$ is given by \eqref{poinc}. Let $S\subset (0,1)$ denote the orbit of $P$ with initial point $s_0\in(0,1]$ -- one of the points on $\{x=1\}$ where $S(x_0)$ intersects it. By Lemma~\ref{fh}, by a local diffeomorphism, the arcs of $S_1(x_0)$ can be transformed to a family $\mathcal{H}_{r,S}=\{(x,y)\in(0,1]\times (0,1]\,:\,x^r y=c,\ c\in S\}$ of countably many hyperbolas. They intersect $\{x=1\}$ at \footnote{By definition, only the density of accumulation of a trajectory on the loop is important for the box dimension of $S(x_0)$; the first \emph{finitely many} windings 
can therefore be \emph{neglected}.}$S$. The box dimension of $\mathcal H_{r,S}$ is given in Lemma~\ref{dim}, using Corollary~\ref{kor} for the box dimension of $S$. Since a local diffeomorphism is a bilipschitz map, we have:
$$
\dim_B(S_1(x_0))=\dim \mathcal{H}_{r,S}=\begin{cases}
2-\frac{2}{k},& \text{$k$ even},\\
2-\frac{2}{k+1},& \text{$k$ odd}.\end{cases}
$$

It is left to compute the dimension of the remaining, regular part $S_2(x_0)$ of the trajectory. In this area, there are no singularities of the vector field. Therefore we directly apply the \emph{Flow-box theorem} \cite{kuzne}, see Section~\ref{one}, and Corollary~\ref{kor}. The box dimension of $S_2(x_0)$ is computed as the box dimension of Cartesian product $S\times[0,\ell],\ \ell>0$: 
$$\dim_B(S_2(x_0))=\begin{cases}
1+(1-\frac{2}{k})=2-\frac{2}{k},& \text{$k$ even},\\
1+(1-\frac{2}{k+1})=2-\frac{2}{k+1},& \text{$k$ odd}.\end{cases}
$$
Finally, by \emph{finite stability property} of the box dimension, the result follows.\hfill $\Box$

\section{Applications}\label{four}

\subsection{The cyclicity of a hyperbolic saddle loop.}\label{fourone}

In recognizing the cyclicity of monodromic limit periodic sets for planar systems, one can use fractal analysis of its trajectories. By \emph{fractal analysis}, we assume analysing the lengths of the $\varepsilon$-neighborhoods of orbits of Poincar\' e maps, as functions of $\varepsilon>0$. Also, analysing the areas of the $\varepsilon$-neighborhoods of trajectories accumulating to the set. The \emph{box dimension} of a set is, by its definition in Section~\ref{one}, related to the asymptotic behavior of the Lebesgue measure of the $\varepsilon$-neighborhood of the set, as $\varepsilon\to 0$. The relation between the box dimension and cyclicity was first investigated in \cite{buletin},\ \cite{belgproc} and later in \cite{cheby}. In \emph{weak focus} and \emph{limit cycle} cases (Poincar\' e maps \emph{differentiable} at $0$) there exists a bijective correspondence between the box dimension of any orbit of the Poincar\' e map and the cyclicity of a set. Further, the bijective correspondence was established between the cyclicity and the box dimension of a spiral trajectory around a focus or a limit cycle, using \emph{Flow-box theorem} or its appropriate adaptation, see \cite{belgproc}. 

The paper \cite{cheby} was motivated by the attempt to do the same for \emph{hyperbolic saddle polycycles}, in particular for the simplest \emph{saddle loop}. Here, the Poincar\' e maps contain logarithmic monomials and are \emph{nondifferentiable} at $0$. Thus, comparison of the asymptotic behavior of the lengths of the $\varepsilon$-neighborhoods to powers of $\varepsilon$ in the very definition of the box dimension is imprecise. See \cite[p. 2502]{cheby} or Example~\ref{staro}\,(1),\,(2) and note that the box dimension reveals the cyclicity ambigously. 

The asymptotic expansion of Poincar\' e map on a transversal through the saddle vertex in Theorem~\ref{saddlepoinc} shows that, although slightly different, it carries exactly the same information as Poincar\' e maps on any other transversal. Further, the box dimension of a two-dimensional spiral trajectory accumulating on the loop exhibits the same deficiency in revealing cyclicity, precisely formulated in Proposition~\ref{ta} below. This is not surprising, due to the fact that Lemmas~\ref{fh} and \ref{dim} develop a version of \emph{Flow-box theorem} adapted for hyperbola-type flow, thus implicitely showing that all dimension information is already contained on a transversal. \newpage

\begin{proposition}[Cyclicity of a saddle loop and the box dimension of its spiral trajectory]\label{ta}
Let $(\Gamma,X_\Lambda)$ be a \emph{generic}\footnote{By \emph{genericity}, we mean the following \emph{regularity condition} on the unfolding from \cite{joyal1} or \cite{mardesic}. Let Poincar\' e maps for the unfolding $(X_\Lambda)$ have the following asymptotic expansion, as $s\to 0$:$$P_\lambda(s)=\alpha_0(\lambda)u_0(s,\lambda)+\ldots+\alpha_k(\lambda)u_k(s,\lambda)+h.o.t.,\ \lambda\in\Lambda\subset\mathbb{R}^n.$$ Suppose $\Gamma$ is of codimension $k\leq N$, $P(s,\lambda_0)=\alpha_k(\lambda_0) u_k(x,\lambda_0)+h.o.t.$ That is, $k$ conditions are imposed on the loop ($\alpha_0(\lambda_0)=\ldots=\alpha_{k-1}(\lambda_0)=0$). The \emph{regularity condition} is that the matrix $$\Big(\frac{\partial\alpha_i(\lambda_0)}{\partial\lambda_j}\Big)_{i=0,\ldots,k-1;\ j=1,\ldots,k}$$ is of maximal rank $k$. That is, by the implicit function theorem, the coefficients $\alpha_0(\lambda),\ldots,\alpha_{k-1}(\lambda)$ can be freely chosen in $X_\Lambda$. Under this assumption on the unfolding, the cyclicity of the loop is exactly equal to its codimension. Without this assumption, it could be smaller.} analytic unfolding of a (monodromic) hyperbolic saddle loop $\Gamma$. Let $S(x_0)$ be a spiral trajectory with initial point $x_0$ accumulating at $\Gamma$, and $$\dim_B(S(x_0))=d\in[1,2).$$ By the box dimension, the cyclicity of the loop is not uniquely determined. More precisely, 
$$
Cycl(\Gamma,X_\Lambda)\in \left \{\frac{2}{2-d}-1,\frac{2}{2-d}\right \}.
$$
\end{proposition}

\begin{proof}
By Theorem~\ref{codim}, a loop $\Gamma$ with $\dim_B(S(x_0))=d$ may be of codimension either $\frac{2}{2-d}-1$ or $\frac{2}{2-d}$. Under genericity assumption on the unfolding, the codimension is equal to the cyclicity.
\end{proof}

\subsection{Prospects: the box dimension of leaves of a foliation of complex resonant saddles}\label{fourtwo}

We consider germs of holomorphic vector fields in $\mathbb{C}^2$, with complex resonant saddles at the origin. That means, with the linear part of the form:
\begin{equation*}
\begin{cases}
\dot{z}=z,\\
\dot{w}=-r\cdot w,\ \ r\in\mathbb{Q}_+^*.
\end{cases}
\end{equation*}
Here, $r\in\mathbb{Q}_+^*$ is the saddle \emph{hyperbolicity ratio}.  Put $r=\frac{p}{q},\ (p,q)=1.$ 

Such germs are either formally orbitally linearizable or formally orbitally equivalent\footnote{Induced foliations are formally conjugated.} to the germ
\begin{equation}\label{haj}
\begin{cases}
\dot z=z,\\
\dot w=w(-r+\frac{u^{k+1}}{1+\lambda u^k}),\ \ u=z^p w^q,
\end{cases}
\end{equation}
for some $k\in\mathbb{N},\ \lambda\in\mathbb{C}$. This result can be looked up in e.g. \cite[Sections 1,\,2,\,22]{ilya}, \cite[Chapters 4,\,5]{loray}, \cite{teyssier}. 

Motivated by \cite{resman}, where the formal type of a germ of an analytic diffeomorphism is read from fractal properties of one orbit, our goal is to see if the \emph{box dimension of a leaf of a foliation} can reveal the \emph{formal type} of the saddle.
\smallskip

Let $L_{a}$ denote a leaf of a foliation through a point $a\in\mathbb{C}^2$ sufficiently close to the saddle. Let $\tau_1=\{w=w_0\}$ and $\tau_2=\{z=z_0\}$ denote a (two-dimensional) horizontal and a vertical cross-section. Let $h_w(z)$ denote the holonomy map induced by $L_a$ on $\tau_1$ and $h_z(w)$ on $\tau_2$. By \cite{ilya}, around each cross-section, a leaf of a foliation has a locally parallel structure (it is a bilipschitz image of a family of unit complex discs).
\smallskip

With respect to formal linearizability, there are two types of resonant saddles, see \cite[Lemma 22.2]{ilya}:
\begin{enumerate}
\item \emph{Formally orbitally linearizable saddles.} It holds that $$h_z^{\circ q}=id,\ h_w^{\circ p}=id.$$
Orbits of holonomy maps on cross-sections consist of finitely many points. The box dimension of orbits is thus equal to $0$. By product structure, we conclude that the box dimension of $L_a$ locally around each cross-section is equal to 2.
\smallskip

\item \emph{Formally orbitally nonlinearizable saddles.} The iterates of the holonomy maps are complex parabolic germs,
$$\ \ h_z^{\circ q}(w)=w+a_1w^{kq}+h.o.t.,\ \ h_w^{\circ p}(z)=z+b_1z^{kp}+h.o.t.$$
Orbits of $h_z,\ h_w$, denoted $S^{h_z}$ and $S^{h_w}$, consist of finitely many ($q,\ p$) disjoint orbits of $h_z^{\circ q}$, $h_w^{\circ p}$ respectively. By \cite{resman} and by finite stability property of box dimension,
$$
\dim_B\left(S^{h_z}\right)=1-\frac{1}{kq+1},\ \ \dim_B\left(S^{h_w}\right)=1-\frac{1}{kp+1}.
$$
Here, $k,\ p,\ q$ are as in \eqref{haj}. By product structure, we conclude that the box dimension of $L_a$ locally around cross-section $\tau_1$ is equal to $3-\frac{1}{kq+1}$. Locally around $\tau_2$, it is equal to $3-\frac{1}{kp+1}$. 
\end{enumerate}
\smallskip

\noindent By monotonicity property of box dimension, under the assumption that box dimension exists, we get:
\begin{equation}\label{bound}
\dim_B(L_a)\geq \begin{cases}
2,& \text{for case $(1)$},\\
\max\left\{3-\frac{1}{kq+1},\ 3-\frac{1}{kp+1}\right\},&\text{for case $(2)$}.
 \end{cases}
\end{equation}

To verify the other side of the inequality \eqref{bound}, we need to compute the box dimension of a leaf in a small neighborhood of the origin, where the product structure is lost. This is left for future research. Motivated by results of Lemma~\ref{dim} for the planar saddle, we conjecture that, also in the complex case, we have equality in \eqref{bound}.

\end{document}